\documentclass[11pt]{article}
\usepackage{amssymb}
\usepackage{latexsym}
\usepackage{amsmath}
\usepackage{amsthm}
\usepackage{amsfonts,pst-node,pst-tree}

\topskip  \theoremstyle{remark}
\theoremstyle{plain}
\newtheorem{theorem}{Theorem}[section]

\newtheorem{definition}[theorem]{Definition}

\begin{document}
%-----------------------------------------------------------
\title{Generalized $q$-Calkin-Wilf trees and $c$-hyper $m$-expansions of integers}
\date{\small }%\today}
\maketitle
\begin{center}Toufik Mansour\\
Department of Mathematics, University of Haifa, 3498838  Haifa, Israel

{\tt tmansour@univ.haifa.ac.il}
\begin{center}Mark Shattuck\\
Department of Mathematics, University of Tennessee, Knoxville, TN 37919

{\tt shattuck@math.utk.edu}
\end{center}

\end{center}

\begin{abstract}
A hyperbinary expansion of a positive integer $n$ is a partition of $n$ into powers of $2$ in which each part appears at most twice.  In this paper, we consider a generalization of this concept and a certain statistic on the corresponding set of expansions of $n$.  We then define $q$-generalized $m$-ary trees whose vertices are labeled by ratios of two consecutive terms within the sequence of distribution polynomials for the aforementioned statistic.  When $m=2$, we obtain a variant of a previously considered $q$-Calkin-Wilf tree.
\end{abstract}



\section{Introduction}

The \emph{Calkin-Wilf tree} (see, e.g., \cite{B,CW}) is a binary tree having root $\frac{1}{1}$ in which a vertex labeled $\frac{a}{b}$ has two children, namely, $\frac{a}{a+b}$ (the left child)
and $\frac{a+b}{b}$ (the right one).  See Figure 1 below.
\begin{figure}[htp]
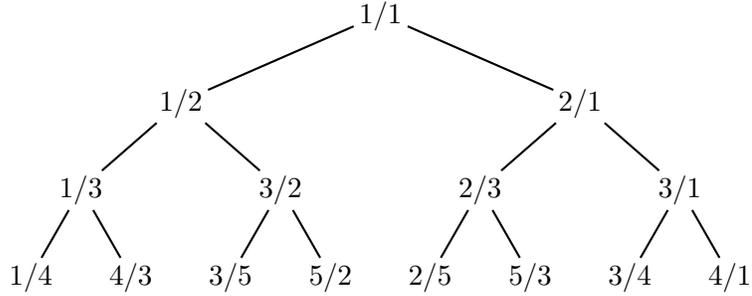

\begin{center}
%\pstree[levelsep=7ex]{\Tcircle[name=R]{11}}{
%            \TC[name=L21]{12}
%            }
%\ncarc[linestyle=dotted,arcangle=260]{->}{R}{L21}\newline
%\psset{arrows=->}
\par
\pstree[nodesep=2pt,levelsep=7ex]{\TR{1/1} }{
        \pstree{
            \TR{1/2} }
                {
                \pstree{\TR{1/3}}
                    {
                    \TR{1/4}
                    \TR{4/3}
                    }
                \pstree{\TR{3/2}}
                    {
                    \TR{3/5}
                    \TR{5/2}
                    }
                }
        \pstree{\TR{2/1}}
              {
              \pstree{\TR{2/3}}
                    {
                    \TR{2/5}
                    \TR{5/3}
                    }
              \pstree{\TR{3/1}}
                    {
                    \TR{3/4}
                    \TR{4/1}
                    }
              }
        }
\end{center}
\caption{The first four levels of the Calkin-Wilf tree.}
\end{figure}
Calkin and Wilf \cite{CW} have shown that each positive
rational number appears exactly once in this tree, as a fraction in lowest terms. The
Calkin-Wilf sequence is obtained by reading the tree line-by-line from left to right. It starts with
\begin{equation*}
\frac{1}{1},\text{ }\frac{1}{2},\text{ }\frac{2}{1},\text{ }\frac{1}{3},%
\text{ }\frac{3}{2},\text{ }\frac{2}{3},\text{ }\frac{3}{1},\text{ }\frac{1}{%
4},\text{ }\frac{4}{3},\text{ }\frac{3}{5},\text{ }\frac{5}{2},\text{ }\frac{%
2}{5},\text{ }\frac{5}{3},\text{ }\frac{3}{4},\text{ }\frac{4}{1},\ldots,
\end{equation*}
and it was found by Newman (see Knuth \cite{New}) that this sequence satisfies the somewhat unusual recurrence
\begin{equation}\label{Newman}
x_{n+1}=\frac{1}{2\left\lfloor x_{n}\right\rfloor +1-x_{n}}, \qquad n \geq 1,
\end{equation}
with initial condition $x_1=1$.  This sequence was investigated as early as 1858 by Stern \cite{Ste} (see also Reznick \cite{Rez} and the references contained therein).
Here, we will consider some related generalized trees that extend a certain aspect of the preceding sequence.

The diatomic sequence $b_n$ is obtained by listing in order the numerators of the terms of the Calkin-Wilf sequence and starts with
\begin{equation*}
1,\text{ }1,\text{ }2,\text{ }1,\text{ }3,\text{ }2,\text{ }3,\text{ }1,\text{ }4,\text{ }3,\text{ }5,\text{ }2,\text{ }5,\text{ }3,\text{ }4,\ldots.
\end{equation*}
It is defined recursively by
$$\quad b_{2n}=b_n,\quad b_{2n+1}=b_n+b_{n+1},\quad n\geq1,$$
with $b_1=1$, and has been an object of recent study (see, for example, \cite{BM,S,U} and the references contained therein).
Various polynomial generalizations \cite{DS,KMP} of the sequence $b_n$ have been considered. For example, Klav$\rm\breve{z}$ar
et al. \cite{KMP} defined the polynomials
\begin{align}\label{eq1}
B_{2n}(t)=tB_n(t),\quad B_{2n+1}(t)=B_n(t)+B_{n+1}(t), \qquad n \geq 1,
\end{align}
with $B_0(t)=0$ and $B_1(t)=1$, and
Dilcher and Stolarsky \cite{DS} defined
\begin{align}\label{eq2}
F_{2n}(q)=F_n(q),\quad F_{2n+1}(q)=qF_n(q)+F_{n+1}(q), \qquad n \geq 1,
\end{align}
with $F_0(q)=F_1(q)=1$. Recently, Mansour \cite{Mqstern} studied a $q$-analogue of the polynomials $B_n(t)$ given by
\begin{align*}
B_{2n}(q,t)=tB_n(q,t),\quad B_{2n+1}(q,t)=qB_n(q,t)+B_{n+1}(q,t), \qquad n \geq 1,
\end{align*}
with the same initial conditions.  Bates and Mansour \cite{BM} used the polynomials defined by \eqref{eq2} to define the $q$-analogue of the Calkin-Wilf tree, and found a statistic on the set of hyperbinary expansions of $n$ for which $F_{n+1}(q)$ is the distribution polynomial.

We will be considering an enumeration related to the following set of expansions of $n$ into powers of a given positive integer $m$.

\begin{definition}
Fix $m\geq2$ and $0\leq c\leq m-1$. By a $c$-hyper $m$-expansion of a positive integer $n$, we mean a partition of $n$ into powers of $m$ in which a given power can appear exactly $j$ times, where $j\in\{0,1,\ldots,m-1,m+c\}$.
\end{definition}

Note that the $m=2, c=0$ case of the preceding definition corresponds to the hyperbinary expansions of $n$.  We consider the following related sequence of polynomials.

\begin{definition}
Given $m\geq2$ and $0\leq c\leq m-1$, define the sequence of polynomials $f_{m,c}(d;q)$ for $d \geq 0$ by
\begin{align}\label{eqff}
\begin{array}{l}
f_{m,c}(mn+j;q)=f_{m,c}(n;q),\qquad j=0,1,\ldots,c-1,c+1,\ldots,m-1,\\
f_{m,c}(mn+c;q)=f_{m,c}(n;q)+qf_{m,c}(n-1;q),
\end{array}
\end{align}
with $f_{m,c}(0;q)=1$ and $f_{m,c}(d;q)=0$ for $d<0$.
\end{definition}

Note that the $f_{m,c}(n;q)$ provide a generalization of the sequence defined by \eqref{eq2} in that $f_{2,0}(n;q)=F_{n+1}(q)$ for all $n \geq 0$.

The presentation of this paper is as follows.  In the next section, we provide combinatorial interpretations for the polynomials $f_{m,c}(n;q)$ in terms of $c$-hyper $m$-expansions of $n$.  In the third section, we describe $m$-ary trees whose vertices are labeled by ratios of consecutive terms of the sequence $f_{m,c}(n;q)$.  Different trees are needed depending on whether $c=m-1$, $c=0$ or $1 \leq c \leq m-2$. When $m=2$, one obtains a variant of the $q$-Calkin-Wilf tree considered in \cite{BM}.  In the case $c=m-1$, the rational functions labeling the vertices of certain branches of the tree may be expressed in terms of Chebyshev polynomials of the second kind.  Furthermore, for all $c$, it is shown that each rational number in the interval $(0,1]$ appears at least once in the corresponding $m$-ary tree when $q=1$.

\section{$(q,c)$-hyper $m$-expansions of the number $n$}

We first define the concept of a $(q,c)$-hyper $m$-expansion of the number $n$.
\begin{definition}
Fix $m \geq 2$ and $0\leq c\leq m-1$. We denote the set of all $c$-hyper $m$-expansions of a positive integer $n$ by $\mathbb{H}_{m,c,n}$ and the number of powers that are used exactly $m+c$ times in the hyper $m$-expansion $x\in\mathbb{H}_{m,c,n}$ by $h_{m,c,n}(x)$. The $(q,c)$-hyper $m$-expansion of $x$ is defined as $q^{h_{m,c,n}(x)}$.
\end{definition}

\begin{definition}
Let $g_{m,c}(n;q)$ be the polynomial consisting of the sum of $(q,c)$-hyper $m$-expansions of $n$, with
$g_{m,c}(0;q)=1$ and $g_{m,c}(r;q)=0$ if $r<0$.
\end{definition}

For example, the $2$-hyper $3$-expansions of $47$ are $27+9+9+1+1$,  $27+9+3+3+1+1+1+1+1$, $9+9+9+9+9+9+1+1$ and $27+3+3+3+3+3+1+1+1+1+1+1$.
Thus, the $(q,2)$-hyper $3$-expansions of $47$ are $q^0$, $q^1$, $q^1$ and
$q^2$ and, accordingly, $g_{3,2}(47;q)=1+2q+q^2$.

\begin{theorem}
For all $n\geq0$, $g_{m,c}(n;q)=f_{m,c}(n;q)$.
\end{theorem}
\begin{proof}
We proceed by induction on $n$. Since $g_{m,c}(0;q)=1=f_{m,c}(0;q)$, the claim holds for $n=0$. Assume that the claim holds for $0,1,\ldots,n-1$ and let us prove it for $n$. By the induction hypothesis and \eqref{eqff}, we have that
\begin{itemize}
\item if $n=mr+j$ with $j\in\{0,1,\ldots,c-1,c+1,\ldots,m-1\}$, then
\begin{align*}
g_{m,c}(n;q)&=g_{m,c}(mr+j;q)=g_{m,c}(r;q)\\
&=f_{m,c}(r;q)=f_{m,c}(mr+j;q)=f_{m,c}(n;q),
\end{align*}

\item if $n=mr+c$, then
\begin{align*}
g_{m,c}(n;q)&=g_{m,c}(mr+c;q)=g_{m,c}(r;q)+qg_{m,c}(r-1;q)\\
&=f_{m,c}(r;q)+qf_{m,c}(r-1;q)=f_{m,c}(mr+c;q)=f_m(n;q).
\end{align*}
\end{itemize}
This completes the induction.
\end{proof}

\section{ $(q,c)$-Calkin-Wilf trees of order $m$}

In this section, we define $m$-ary trees whose vertices are labeled with certain rational functions of $q$, in particular, by ratios of consecutive terms of the $g_{m,c}(n;q)$ sequence.

\subsection{Case $c=m-1$}

In this subsection, we construct an $m$-ary tree whose vertices are labeled by the ratios of certain terms within the $g_{m,c}(n;q)$ sequence in the case when $c=m-1$ and consider some of its properties.

\begin{definition}\label{defqmt}
If $m \geq 3$, then the $(q,m-1)$-Calkin-Wilf tree of order $m$ is an $m$-ary tree with root $\frac{1}{1}$. A vertex labeled $\frac{a}{b}$ is a parent of $m$ children defined, from left to right, as follows. Each of the first $m-2$ children is $\frac{1}{1+q}$, with the $(m-1)$-st child given by $\frac{b}{b+qa}$.  To define the $m$-th child, suppose $\frac{a_j}{b_j}$ is the $(m-1)$-st child of $\frac{a_{j+1}}{b_{j+1}}$ for $j=1,2,\ldots,s-1$, where $\frac{a_1}{b_1}=\frac{a}{b}$ and $s\geq 1$ is maximal.  Then the $m$-th child of $\frac{a}{b}$ is given by $\frac{1}{1+qp\prod_{j=1}^{s}\frac{b_j}{a_j}}$, where $p=\frac{a_{s+1}}{b_{s+1}}$ if $\frac{a_s}{b_s}$ is the $(m-2)$-nd child of $\frac{a_{s+1}}{b_{s+1}}$ and $p=1$ otherwise.
\end{definition}

The following figure illustrates the $(q,2)$-Calkin-Wilf tree of order $3$ when $q=1$.\pagebreak

\begin{figure}[htp]
\begin{center}
\pstree[nodesep=2pt,levelsep=6ex,treesep=5pt]{\TR{$\frac11$} }{
        \pstree{\TR{$\frac12$}}
                {
                \pstree{\TR{$\frac12$}}
                    {
                    \TR{$\frac12$}
                    \TR{$\frac23$}
                    \TR{$\frac12$}
                    }
                \pstree{\TR{$\frac23$}}
                    {
                    \TR{$\frac12$}
                    \TR{$\frac35$}
                    \TR{$\frac14$}
                    }
                \pstree{\TR{$\frac13$}}
                    {
                    \TR{$\frac12$}
                    \TR{$\frac34$}
                    \TR{$\frac14$}
                    }
                }
        \pstree{\TR{$\frac12$}}
              {
               \pstree{\TR{$\frac12$}}
                    {
                    \TR{$\frac12$}
                    \TR{$\frac23$}
                    \TR{$\frac12$}
                    }
               \pstree{\TR{$\frac23$}}
                    {
                    \TR{$\frac12$}
                    \TR{$\frac35$}
                    \TR{$\frac14$}
                    }
               \pstree{\TR{$\frac13$}}
                    {
                    \TR{$\frac12$}
                    \TR{$\frac34$}
                    \TR{$\frac14$}
                    }
              }
        \pstree{\TR{$\frac12$}}
              {
              \pstree{\TR{$\frac12$}}
                    {
                    \TR{$\frac12$}
                    \TR{$\frac23$}
                    \TR{$\frac12$}
                    }
              \pstree{\TR{$\frac23$}}
                    {
                    \TR{$\frac12$}
                    \TR{$\frac35$}
                    \TR{$\frac14$}
                    }
              \pstree{\TR{$\frac13$}}
                    {
                    \TR{$\frac12$}
                    \TR{$\frac34$}
                    \TR{$\frac14$}
                    }
               }
        }
\end{center}
\caption{The first four levels of the $(q,2)$-Calkin-Wilf tree of order $3$ with $q=1$.}\label{fig3e}
\end{figure}

\begin{theorem}
Let $m\geq3$ and let the concatenation of successive levels of the $(q,m-1)$-Calkin-Wilf tree of order $m$ form a sequence $\{\ell_m(n;q)\}_{n\geq0}$. Then
$$\ell_m(n;q)=\frac{g_{m,m-1}(mn+m-2;q)}{g_{m,m-1}(mn+m-1;q)},$$
for all $n\geq0$.
\end{theorem}
\begin{proof}
Let $\mathcal{T}_m$ be the $(q,m-1)$-Calkin-Wilf tree of order $m$, where $m \geq 3$.
We proceed by induction on $n$. Since $\ell_m(0,q)=\frac{g_{m,m-1}(m-2;q)}{g_{m,m-1}(m-1;q)}=\frac{1}{1}$, the claim holds for $n=0$.
Assume that the claim holds for $\ell_m(0;q),\ell_m(1;q),\ldots,\ell_m(n;q)$ of $\mathcal{T}_m$, and let us prove it for the children of $\ell_m(n;q)$. One can verify that the children of  $\ell_m(n;q)$ are $\ell_m(mn+1;q),\ell_m(mn+2;q),\ldots,\ell_m(mn+m;q)$. By the induction hypothesis and Definition \ref{defqmt}, we have
\begin{align*}
\frac{g_{m,m-1}(m(mn+j)+m-2;q)}{g_{m,m-1}(m(mn+j)+m-1;q)}&=\frac{g_{m,m-1}(mn+j;q)}{g_{m,m-1}(mn+j;q)+qg_{m,m-1}(mn+j-1;q)}\\
&=\frac{g_{m,m-1}(n;q)}{g_{m,m-1}(n;q)+qg_{m,m-1}(n;q)}\\
&=\frac{1}{1+q}=\ell_m(mn+j;q),
\end{align*}
for all $j \in [m-2]=\{1,2,\ldots,m-2\}$. Thus, the claim holds for $\ell_m(mn+j;q)$ when $j\in[m-2]$.
Also, we have
\begin{align*}
&\frac{g_{m,m-1}(m(mn+m-1)+m-2;q)}{g_{m,m-1}(m(mn+m-1)+m-1;q)}\\
&\qquad\qquad=\frac{g_{m,m-1}(mn+m-1;q)}{g_{m,m-1}(mn+m-1;q)+qg_{m,m-1}(mn+m-2;q)}\\
&\qquad\qquad=\frac{1}{1+q\ell_m(n;q)}=\ell_m(mn+m-1;q),
\end{align*}
which implies that the claim holds for $\ell_m(mn+m-1;q)$. Thus, it remains to show that  $\frac{g_{m,m-1}(m(mn+m)+m-2;q)}{g_{m,m-1}(m(mn+m)+m-1;q)}=\ell_m(mn+m;q)$. To do so, let $n_1=n$ and $n_j=mn_{j+1}+m-1$ for $j=1,2,\ldots,s-1$, with $s\geq1$ maximal. Thus, either $n_s=mn'+j$ for some $j\in\{1,2,\ldots,m-2,m\}$ or $n_s=0$. By the definitions, we have
\begin{align*}
&\frac{g_{m,m-1}(m(mn+m)+m-2;q)}{g_{m,m-1}(m(mn+m)+m-1;q)}\\
\qquad\qquad&=\frac{g_{m,m-1}(mn+m;q)}{g_{m,m-1}(mn+m;q)+qg_{m,m-1}(mn+m-1;q)}\\
\qquad\qquad&=\frac{g_{m,m-1}(n+1;q)}{g_{m,m-1}(n+1;q)+q(g_{m,m-1}(n;q)+qg_{m,m-1}(n-1;q))}\\
\qquad\qquad&=\frac{1}{1+q\frac{g_{m,m-1}(n;q)+qg_{m,m-1}(n-1;q)}{g_{m,m-1}(n;q)}\frac{g_{m,m-1}(n;q)}{g_{m,m-1}(n+1;q)}},
\end{align*}
with
\begin{align*}
\frac{g_{m,m-1}(n_1;q)}{g_{m,m-1}(n_1+1;q)}&=\frac{g_{m,m-1}(mn_2+m-1;q)}{g_{m,m-1}(mn_2+m;q)}\\
&=\frac{g_{m,m-1}(n_2;q)+qg_{m,m-1}(n_2-1;q)}{g_{m,m-1}(n_2+1;q)}\\
&=\frac{g_{m,m-1}(n_2;q)+qg_{m,m-1}(n_2-1;q)}{g_{m,m-1}(n_2;q)}\frac{g_{m,m-1}(n_2;q)}{g_{m,m-1}(n_2+1;q)}\\
&=\cdots\\
&=\frac{g_{m,m-1}(n_s;q)}{g_{m,m-1}(n_s+1;q)}\prod_{j=2}^s\frac{g_{m,m-1}(n_j;q)+qg_{m,m-1}(n_j-1;q)}{g_{m,m-1}(n_j;q)}.
\end{align*}
Note that if $j\in[m-3]$ of if $n_s=0$, then
$$\frac{g_{m,m-1}(n_s;q)}{g_{m,m-1}(n_s+1;q)}=\frac{g_{m,m-1}(n';q)}{g_{m,m-1}(n';q)}=1.$$
If $j=m$, then
$$\frac{g_{m,m-1}(n_s;q)}{g_{m,m-1}(n_s+1;q)}=\frac{g_{m,m-1}(n'+1;q)}{g_{m,m-1}(n'+1;q)}=1.$$
If $j=m-2$, then
$$\frac{g_{m,m-1}(n_s;q)}{g_{m,m-1}(n_s+1;q)}=\frac{g_{m,m-1}(mn'+m-2;q)}{g_{m,m-1}(mn'+m-1;q)}=\ell_m(n';q).$$
Hence,
\begin{align*}
\frac{g_{m,m-1}(n_1;q)}{g_{m,m-1}(n_1+1;q)}
&=\frac{g_{m,m-1}(n_s;q)}{g_{m,m-1}(n_s+1;q)}\prod_{j=2}^s\frac{g_{m,m-1}(n_j;q)+qg_{m,m-1}(n_j-1;q)}{g_{m,m-1}(n_j;q)}\\
&=p\prod_{j=2}^s\frac{g_{m,m-1}(n_j;q)+qg_{m,m-1}(n_j-1;q)}{g_{m,m-1}(n_j;q)},
\end{align*}
where $p$ is as defined above.  Thus, by the induction hypothesis and the definitions, we have
\begin{align*}
\frac{g_{m,m-1}(m(mn+m)+m-2;q)}{g_{m,m-1}(m(mn+m)+m-1;q)}&=\frac{1}{1+qp\prod_{j=1}^s\frac{g_{m,m-1}(n_j;q)+qg_{m,m-1}(n_j-1;q)}{g_{m,m-1}(n_j;q)}}\\
&=\frac{1}{1+qp\prod_{j=1}^s\frac{g_{m,m-1}(mn_j+m-1;q)}{g_{m,m-1}(mn_j+m-2;q)}}\\
&=\frac{1}{1+qp\prod_{j=1}^s\frac{1}{\ell_m(n_j;q)}}\\
&=\ell_m(mn+m;q).
\end{align*}
Hence, the children of $\ell_m(n;q)=\frac{g_{m,m-1}(mn+m-2;q)}{g_{m,m-1}(mn+m-1;q)}$ in $\mathcal{T}_m$ are $$\left\{\ell_m(mn+j;q)=\frac{g_{m,m-1}(m(mn+j)+m-2;q)}{g_{m,m-1}(m(mn+j)+m-1;q)}\right\}_{j=1}^m,$$ which completes the induction.
\end{proof}

\emph{Remark:} The result of the previous theorem also holds when $m=2$, provided that the $p$ factor in Definition \ref{defqmt} is adjusted in the case $m=2$ as follows: let $p=\ell_2(n'+1;q)$ if $n_s=2n'+2$, where $n_s$ and $n'$ are as defined in the prior proof, and $p=1$ otherwise.

\begin{definition} Let $v$ be any vertex of the $(q,m-1)$-Calkin-Wilf tree of order $m$ and let $j\in[m]$. The set of all vertices $v_1=v,v_2,v_3,\ldots$ such that $v_i$ is the $j$-th child of $v_{i-1}$ for $i\geq2$ will be denoted by $B_{v,j}$ and will be called the $j$-th branch of $v$.
\end{definition}

For example, the $m$-th branch of the root $1/1$  is given by
$B_{1/1,m}=\{1/1,1/(1+q),1/(1+q+q^2),1/(1+q+q^2+q^3),\ldots\}$ for $m \geq 3$.

In order to state our next result, we recall the Chebyshev polynomials of the second kind (see \cite{Ri}) defined by the recurrence  \begin{align}U_n(t)=2tU_{n-1}(t)-U_{n-2}(t), \qquad n \geq 2,\label{equu}
\end{align}
with $U_0(t)=1$ and $U_1(t)=2t$.

\begin{theorem}
The $(m-1)$-st branch of the root in the $(q,m-1)$-Calkin-Wilf tree of order $m\geq 2$ is given by
$$B_{1/1,m-1}=\left\{\frac{U_j\left(\frac{1}{2\sqrt{-q}}\right)}{\sqrt{-q}U_{j+1}\left(\frac{1}{2\sqrt{-q}}\right)}\right\}_{j\geq0}.$$
If $m\geq 3$, then the $j$-th branch of the root is given by
$B_{1/1,j}=\{1/1,1/(1+q),1/(1+q),1/(1+q),\ldots\}$ for $j \in [m-2]$, and the $m$-th branch is given by
$B_{1/1,m}=\{1/1,1/(1+q),1/(1+q+q^2),1/(1+q+q^2+q^3),\ldots\}$.
\end{theorem}
\begin{proof}
Let $x_n=\frac{g_{m,m-1}(mn+m-2;q)}{g_{m,m-1}(mn+m-1;q)}$ for $n \geq0$. By induction, one can show that the $(m-1)$-st branch is given by
$B_{1/1,m-1}=\left\{x_{m^j-1}\right\}_{j\geq0}$.
By \eqref{eqff}, we have
\begin{align*}
x_{m^j-1}&=\frac{g_{m,m-1}(m(m^j-1)+m-2;q)}{g_{m,m-1}(m(m^j-1)+m-1;q)}\\
&=\frac{g_{m,m-1}(m^j-1;q)}{g_{m,m-1}(m^j-1;q)+qg_{m,m-1}(m^j-2;q)}\\
&=\frac{g_{m,m-1}(m(m^{j-1}-1)+m-1;q)}{g_{m,m-1}(m(m^{j-1}-1)+m-1;q)+qg_{m,m-1}(m(m^{j-1}-1)+m-2;q)}\\
&=\frac{1}{1+qx_{m^{j-1}-1}}, \qquad j \geq 1,
\end{align*}
with $x_0=1$.  By induction on $j$ and \eqref{equu}, we obtain
$$x_{m^j-1}=\frac{U_j\left(\frac{1}{2\sqrt{-q}}\right)}{\sqrt{-q}U_{j+1}\left(\frac{1}{2\sqrt{-q}}\right)}, \qquad j \geq 0.$$
The second statement follows from the definitions.
\end{proof}

Similarly, one can show the following result.

\begin{theorem}
Let $v=\frac{U_j\left(\frac{1}{2\sqrt{-q}}\right)}{\sqrt{-q}U_{j+1}\left(\frac{1}{2\sqrt{-q}}\right)}$. Then
the $(m-1)$-st branch of $v$ in the $(q,m-1)$-Calkin-Wilf tree of order $m\geq2$ is given by
$$B_{v,m-1}=\left\{\frac{U_i\left(\frac{1}{2\sqrt{-q}}\right)}{\sqrt{-q}U_{i+1}\left(\frac{1}{2\sqrt{-q}}\right)}\right\}_{i\geq j}.$$
If $m \geq 3$, then the $j$-th branch of $v$ is given by
$B_{v,j}=\{v,1/(1+q),1/(1+q),1/(1+q),\ldots\}$ for $j \in [m-2]$, and the $m$-th branch is given by
$B_{v,m}=\{v,1/(1+qt),1/(1+q+q^2t),1/(1+q+q^2+q^3t),\ldots\}$, where $t=\sqrt{-q}^{j+1}U_{j+1}\left(\frac{1}{2\sqrt{-q}}\right)$.
\end{theorem}

We conclude the case $c=m-1$ with the following result when $q=1$ concerning the tree $\mathcal{T}_m$.

\begin{theorem}
When $q=1$, each positive rational number less than or equal one appears at least once in the $(q,m-1)$-Calkin-Wilf tree of order $m$ for all $m\geq 3$.
\end{theorem}
\begin{proof}
We will show that all rational numbers $\frac{a}{b}$, where $0<\frac{a}{b}\leq 1$, belong to $\mathcal{T}_m$ when $m \geq 3$ and $q=1$  First note that all fractions of the form $\frac{1}{b}$ belong to $\mathcal{T}_m$, upon considering the $m$-th branch of the root.

So suppose $0<\frac{a}{b}<1$ is a rational number (in lowest terms) such that $a>1$.  Let $\frac{1}{x}$ be an element of the $m$-th branch of the root, where $x>1$ is to be determined.  Let $v=\frac{1}{2}$ be the $(m-2)$-nd child of $\frac{1}{x}$.  Consider the $(m-1)$-st branch of $v$, the sequence of which we will denote by $v_1=v,v_2,v_3,\ldots$.  Let $v_i=\frac{a_i}{b_i}$ in reduced form for $i \geq 1$.  It can be shown by induction that $a_i=f_i$ and $b_i=f_{i+1}$, where $f_i$ denotes the Fibonacci sequence defined by $f_i=f_{i-1}+f_{i-2}$ for $i\geq 2$ with $f_0=f_1=1$.  If $t\geq 1$, then $\prod_{i=1}^t\frac{b_i}{a_i} = f_{t+1}$, which implies that the $m$-th child of $v_t$ is $\frac{1}{1+\frac{f_{t+1}}{x}}$, by the definitions.

Recall the well known fact (see, e.g., \cite[p. 73-74]{Le}) that given any positive integer $j$, there exists some $k$ such that $j$ divides $f_k$.  Choosing $t$ so that $b-a$ divides $f_{t+1}$, i.e., $f_{t+1}=u(b-a)$ for some $u \geq 1$, and then letting $x=ua$, implies that the $m$-th child of $v_t$ is given by
$$\frac{1}{1+\frac{f_{t+1}}{x}}=\frac{1}{1+\frac{u(b-a)}{ua}}=\frac{a}{b}.$$
Thus, we have $\frac{a}{b} \in \mathcal{T}_m$, which completes the proof.
\end{proof}

\subsection{Case $c=0$}

A comparable tree may be constructed in the case when $c=0$.

\begin{definition}\label{defqm0}
The $(q,0)$-Calkin-Wilf tree of order $m$ is an $m$-ary tree with root $\frac{1}{1}$. A vertex labeled $\frac{a}{b}$ is a parent to $m$ children defined, from left to right, as follows.  Each of the first $m-2$ children is $\frac{1}{1+q}$, with the $(m-1)$-st child given by $\frac{a}{b+qa}$.  To define the $m$-th child, suppose $\frac{a_j}{b_j}$ is the $m$-th child of $\frac{a_{j+1}}{b_{j+1}}$ for $j=1,2,\ldots,s-1$, where $\frac{a_1}{b_1}=\frac{a}{b}$ and $s\geq 1$ is maximal.  Then the $m$-th child of $\frac{a}{b}$ is given by $$\frac{b}{qb+\frac{a}{1+q+\cdots+q^{s-2}+q^{s-1}r}},$$ where $r=\frac{a_{s+1}}{b_{s+1}}$ if $\frac{a_s}{b_s}$ is the $(m-1)$-st child of $\frac{a_{s+1}}{b_{s+1}}$ and $r=1$ otherwise.
\end{definition}

One can describe an $m$-ary tree in analogy to the case $c=m-1$ above whose vertices are labeled by rational functions of the form $\frac{g_{m,0}(mn+m-1;q)}{g_{m,0}(mn+m;q)}$.

\begin{theorem}
Let $m\geq2$ and let the concatenation of successive levels of the $(q,0)$-Calkin-Wilf tree of order $m$ form a sequence $\{\ell_m(n;q)\}_{n\geq0}$. Then
$$\ell_m(n;q)=\frac{g_{m,0}(mn+m-1;q)}{g_{m,0}(mn+m;q)},$$
for all $n\geq0$.
\end{theorem}
\begin{proof}
Let $\mathcal{T}_m$ be the $(q,0)$-Calkin-Wilf tree of order $m$.
We proceed by induction on $n$. Since $\ell_m(0,q)=\frac{g_{m,0}(m-1;q)}{g_{m,0}(m;q)}=\frac{1}{1}$, the claim holds for $n=0$.
Assume that the claim holds for $\ell_m(0;q),\ell_m(1;q),\ldots,\ell_m(n;q)$ of $\mathcal{T}_m$ and let us prove it for the children of $\ell_m(n;q)$, which are $\ell_m(mn+1;q),\ell_m(mn+2;q),\ldots,\ell_m(mn+m;q)$. By the induction hypothesis and Definition \ref{defqm0}, we have
\begin{align*}
\frac{g_{m,0}(m(mn+j)+m-1;q)}{g_{m,0}(m(mn+j)+m;q)}&=\frac{g_{m,0}(mn+j;q)}{g_{m,0}(mn+j+1;q)+qg_{m,0}(mn+j;q)}\\
&=\frac{g_{m,0}(n;q)}{g_{m,0}(n;q)+qg_{m,0}(n;q)}\\
&=\frac{1}{1+q}=\ell_m(mn+j;q),
\end{align*}
for all $j\in[m-2]$. Thus, the claim holds for $\ell_m(mn+j;q)$ when $j\in[m-2]$.
We also have
\begin{align*}
\frac{g_{m,0}(m(mn+m-1)+m-1;q)}{g_{m,0}(m(mn+m-1)+m;q)}&=\frac{g_{m,0}(mn+m-1;q)}{g_{m,0}(mn+m;q)+qg_{m,0}(mn+m-1;q)}\\
&=\frac{1}{q+\frac{g_{m,0}(mn+m;q)}{g_{m,0}(mn+m-1;q)}}=\frac{1}{q+\frac{1}{\ell_{m}(n;q)}}\\
&=\ell_m(mn+m-1;q),
\end{align*}
which implies that the claim holds for $\ell_m(mn+m-1;q)$. Thus, it remains to show that  $\frac{g_{m,0}(m(mn+m)+m-1;q)}{g_{m,0}(m(mn+m)+m;q)}=\ell_m(mn+m;q)$.
Let $n_1=n$ and $n_j=mn_{j+1}+m$ for $j=1,2,\ldots,s-1$, with $s$ maximal. Thus, $n_s=mn'+j$ for some $j\in[m-1]$ or $n_s=0$. By the definitions, we have
\begin{align*}
\frac{g_{m,0}(m(mn+m)+m-1;q)}{g_{m,0}(m(mn+m)+m;q)}&=\frac{g_{m,0}(mn+m;q)}{g_{m,0}(mn+m+1;q)+qg_{m,0}(mn+m;q)}\\
&=\frac{1}{q+\frac{g_{m,0}(mn+m+1;q)}{g_{m,0}(mn+m;q)}}\\
&=\frac{1}{q+\ell_m(n;q)\frac{g_{m,0}(n+1;q)}{g_{m,0}(n;q)}},
\end{align*}
where
\begin{align*}
\frac{g_{m,0}(n_1+1;q)}{g_{m,0}(n_1;q)}&=\frac{g_{m,0}(mn_2+m+1;q)}{g_{m,0}(mn_2+m;q)}=\frac{g_{m,0}(n_2+1;q)}{g_{m,0}(n_2+1;q)+qg_{m,0}(n_2;q)}\\
&=\frac{1}{1+q\frac{g_{m,0}(n_2;q)}{g_{m,0}(n_2+1;q)}}=\frac{1}{1+q\frac{g_{m,0}(mn_3+m;q)}{g_{m,0}(mn_3+m+1;q)}}\\
&=\frac{1}{1+q\frac{g_{m,0}(n_3+1;q)+qg_{m,0}(n_3;q)}{g_{m,0}(n_3+1;q)}}=\frac{1}{1+q+q^2\frac{g_{m,0}(n_3;q)}{g_{m,0}(n_3+1;q)}}\\
&=\cdots\\
&=\frac{1}{1+q+\cdots+q^{s-2}+q^{s-1}\frac{g_{m,0}(n_s;q)}{g_{m,0}(n_s+1;q)}}\\
&=\frac{1}{1+q+\cdots+q^{s-2}+q^{s-1}\frac{g_{m,0}(mn'+j;q)}{g_{m,0}(mn'+j+1;q)}}.
\end{align*}
Note that when $j\in[m-2]$, we have $\frac{g_{m,0}(mn'+j;q)}{g_{m,0}(mn'+j+1;q)}=1$, and when $j=m-1$, we have
$\frac{g_{m,0}(mn'+j;q)}{g_{m,0}(mn'+j+1;q)}=\ell_m(n';q)$.  It follows from the definitions that
\begin{align*}
\frac{g_{m,0}(m(mn+m)+m-1;q)}{g_{m,0}(m(mn+m)+m;q)}&=\ell_m(mn+m;q),
\end{align*}
which completes the induction.
\end{proof}

We have the following result when $q=1$ concerning the tree $\mathcal{T}_m$.

\begin{theorem}
When $q=1$, each positive rational number less than or equal one appears at least once in the $(q,0)$-Calkin-Wilf tree of order $m$ for all $m\geq 2$.
\end{theorem}
\begin{proof}
We will show that all rational numbers $\frac{a}{b}$ in lowest terms, where $0 < \frac{a}{b} \leq 1$, belong to $\mathcal{T}_m$ when $q=1$ by inducting on the sum $s=a+b$, the case $s=2$ clear.  First suppose $0 < \frac{a}{b}\leq \frac{1}{2}$.  Then $\frac{a}{b-a} \in \mathcal{T}_m$, by hypothesis, and has $(m-1)$-st child $\frac{a}{b}$, which implies $\frac{a}{b} \in \mathcal{T}_m$.

So assume $\frac{1}{2}<\frac{a}{b}<1$.  We will construct a vertex whose label is $\frac{a}{b}$.  To do so, let $v$ be a vertex of $\mathcal{T}_m$ labeled by $\frac{x}{y}$, where $0<\frac{x}{y}\leq\frac{1}{2}$ is in lowest terms and $x$ and $y$ are to be determined.  Note that $v$ can be taken to be an $(m-1)$-st child of a vertex labeled by $\frac{x}{y-x}$.  Consider the sequence $v=v_0,v_1,v_2,\ldots$ of vertices of $\mathcal{T}_m$ such that $v_i$ is the $m$-th child of $v_{i-1}$ for $i \geq 1$.  Using the definitions, one can show by induction that the vertex $v_i$ is labeled by $\frac{iy-(i-1)x}{(i+1)y-ix}$ for all $i \geq 0$.  Note that $x$ and $y$ relatively prime implies that each of these fractions is in lowest terms.

Suppose now that $j\geq 1$ is determined by the condition $\frac{j}{j+1}<\frac{a}{b}\leq \frac{j+1}{j+2}$.  Setting
$\frac{a}{b}=\frac{jy-(j-1)x}{(j+1)y-jx}$ implies $x=(j+1)a-jb$ and $y=ja-(j-1)b$.  Note that $a$ and $b$ relatively prime implies $x$ and $y$ are.  Furthermore, using the restrictions on $\frac{a}{b}$, one can show that $0<\frac{x}{y}\leq \frac{1}{2}$, as required.  Finally, note that $x+y=(2j+1)a-(2j-1)b<a+b$, which implies $\frac{x}{y} \in \mathcal{T}_m$, by hypothesis.  Thus, taking the $m$-th child exactly $j$ times starting with any vertex labeled by $\frac{x}{y}$ implies $\frac{a}{b} \in \mathcal{T}_m$, which completes the induction.
\end{proof}

\subsection{Case $1\leq c\leq m-2$}

The remaining cases when $1 \leq c \leq m-2$ may be described in terms of a single tree.

\begin{definition}\label{defqmtc}
Given $m \geq 3$ and $1 \leq c \leq m-2$, the $(q,c)$-Calkin-Wilf tree of order $m$ is an $m$-ary tree with root $\frac{1}{1}$. Each vertex labeled $\frac{a}{b}$ is a parent to $m$ children defined, from left to right, as follows: the $k$-th child for $k\neq c,c+1$ is $\frac{1}{1+q}$, the $c$-th child is $\frac{1}{1+q\frac{a}{b}}$, and the $(c+1)$-st child is $\frac{1}{1+q\frac{b}{a}}$.
\end{definition}

\begin{theorem}
Let $m\geq3$ and $1\leq c\leq m-2$. Suppose that the concatenation of successive levels of the $(q,c)$-Calkin-Wilf tree of order $m$ forms a sequence $\{\ell_{m,c}(n;q)\}_{n\geq0}$. Then
$$\ell_{m,c}(n;q)=\frac{g_{m,c}(mn+c-1;q)}{g_{m,c}(mn+c;q)},$$
for all $n\geq0$.
\end{theorem}
\begin{proof}
Let $\mathcal{T}_{m,c}$ be the $(q,c)$-Calkin-Wilf tree of order $m$.
We proceed by induction on $n$, the $n=0$ case clear.  We again prove the claim for the children of $\ell_{m,c}(n;q)$. By the induction hypothesis and Definition \ref{defqmtc}, we have
\begin{align*}
\frac{g_{m,c}(m(mn+j)+c-1;q)}{g_{m,c}(m(mn+j)+c;q)}&=\frac{g_{m,c}(mn+j;q)}{g_{m,c}(mn+j;q)+qg_{m,c}(mn+j-1;q)}\\
&=\frac{g_{m,c}(n;q)}{g_{m,c}(n;q)+qg_{m,c}(n;q)}\\
&=\frac{1}{1+q}=\ell_m(mn+j;q),
\end{align*}
for all $j\in [m]$ and $j\neq c,c+1$.
When $j=c$, we have
\begin{align*}
\frac{g_{m,c}(m(mn+c)+c-1;q)}{g_{m,c}(m(mn+c)+c;q)}&=\frac{g_{m,c}(mn+c;q)}{g_{m,c}(mn+c;q)+qg_{m,c}(mn+c-1;q)}\\
&=\frac{1}{1+q\frac{g_{m,c}(mn+c-1;q)}{g_{m,c}(mn+c;q)}}\\
&=\frac{1}{1+q\ell_{m,c}(n;q)}\\
&=\ell_{m,c}(mn+c;q),
\end{align*}
which implies that the claim holds for $\ell_{m,c}(mn+c;q)$. Finally, when $j=c+1$, we have
\begin{align*}
\frac{g_{m,c}(m(mn+c+1)+c-1;q)}{g_{m,c}(m(mn+c+1)+c;q)}&=\frac{g_{m,c}(mn+c+1;q)}{g_{m,c}(mn+c+1;q)+qg_{m,c}(mn+c;q)}\\
&=\frac{1}{1+q\frac{g_{m,c}(mn+c;q)}{g_{m,c}(mn+c+1;q)}}\\
&=\frac{1}{1+q\frac{g_{m,c}(mn+c;q)}{g_{m,c}(mn+c-1;q)}}\\
&=\frac{1}{1+q\frac{1}{\ell_{m,c}(n;q)}}\\
&=\ell_{m,c}(mn+c+1;q),
\end{align*}
which implies that the claim holds for $\ell_{m,c}(mn+c+1;q)$ and completes the induction.
\end{proof}

\emph{Remark:}  When $q=1$, each positive rational number less than or equal $1$ appears at least once in $\mathcal{T}_{m,c}$ as a fraction in lowest terms for all $m \geq 3$ and $1 \leq c \leq m-2$.  This follows from the definitions, upon inducting on the sum $a+b$ corresponding to a vertex labeled by the fraction $\frac{a}{b}$ in lowest terms.  Indeed, each rational in the interval $(0,1)$ is seen to occur infinitely many times in the tree since it essentially starts over  each time a vertex is labeled by $\frac{1}{2}$.

%-------------------------

\end{document}